\documentclass[12pt,reqno]{amsart}
\usepackage{graphicx}
\usepackage{psfrag}
\usepackage{mathrsfs}
\usepackage{color}
\usepackage{amsmath,amssymb}
\font\de=cmssi12

\numberwithin{equation}{section}

\newcommand{\graph}{\operatorname{Graph}}

\newcommand{\eps} {\varepsilon}
\def \cT {{\mathcal T}}

\def \cX {{\mathcal X}}

\newtheorem{theorem}{Theorem}[section]

\newtheorem{proposition}[theorem]{Proposition}
\newtheorem{lemma}[theorem]{Lemma}
\newtheorem{definition}[theorem]{Definition}

\theoremstyle{remark}
\newtheorem{remark}[theorem]{Remark}

\def\fp{\noindent{\hfill $\fbox{\,}$}\medskip\newline}

\begin{document}

%\thanks{ }

\author{F. Rodriguez Hertz}
\address{IMERL-Facultad de Ingenier\'\i a\\ Universidad de la
Rep\'ublica\\ CC 30 Montevideo, Uruguay.}
\email{frhertz@fing.edu.uy}\urladdr{http://www.fing.edu.uy/$\sim$frhertz}

\author{M. A. Rodriguez Hertz}
\address{IMERL-Facultad de Ingenier\'\i a\\ Universidad de la
Rep\'ublica\\ CC 30 Montevideo, Uruguay.}
\email{jana@fing.edu.uy}\urladdr{http://www.fing.edu.uy/$\sim$jana}

\author{A. Tahzibi}
\address{Departamento de Matem\'atica,
  ICMC-USP S\~{a}o Carlos, Caixa Postal 668, 13560-970 S\~{a}o
  Carlos-SP, Brazil.}
\email{tahzibi@icmc.sc.usp.br}\urladdr{http://www.icmc.sc.usp.br/$\sim$tahzibi}

\author{R. Ures}
\address{IMERL-Facultad de Ingenier\'\i a\\ Universidad de la
Rep\'ublica\\ CC 30 Montevideo, Uruguay.} \email{ures@fing.edu.uy}
\urladdr{http://www.fing.edu.uy/$\sim$ures}


\renewcommand{\subjclassname}{\textup{2000} Mathematics Subject Classification}

\date{\today}

\setcounter{tocdepth}{2}

\title{Creation of blenders in the conservative setting}
\maketitle
\begin{abstract}
In this work we prove that each $C^r$ conservative diffeomorphism with a pair of hyperbolic periodic points of co-index one can be $C^1$-approximated by $C^r$ conservative diffeomorphisms having a blender.
\end{abstract}
%----------------------------------------------------------------------------------
\section{Introduction}
%----------------------------------------------------------------------------------

One major task in the theory of dynamics is to establish some kind of dynamic irreducibility
of a system. Of principal interest are the systems that display some kind of persistent irreducibility.
The two main examples of this concept are robust transitivity and stable ergodicity.\par
Blenders were introduced by C. Bonatti and L. D\'{\i}az in \cite{bonattidiaz1996}, to produce a large class of examples of non-hyperbolic robustly transitive diffeomorphisms. There, they showed that these objects appeared in a neighborhood of the time-one map of any transitive Anosov flow. These systems are partially hyperbolic. To establish robust transitivity, they showed that the strong invariant manifolds entered a small ball (the {\em blender}), where things got mixed. This blending was then distributed all over the manifold by means of the strong invariant manifolds. This phenomenon is robust, whence they get robust transitivity. \par
In \cite{bonatti-diaz2006}, C. Bonatti and L. D\'{\i}az showed that blenders appear near co-index one heterodimensional cycles. This provides a local source of robust transitivity. \par
In \cite[Theorem C]{rhrhtu2009}, the authors showed that, surprisingly, blenders also provide a local source of stable ergodicity. This arouses some interest in the appearance of blenders in the conservative setting. Indeed, the presence of blenders near pairs of periodic points of co-index one in the conservative setting, allows the authors to prove a special case of a longstanding conjecture by C. Pugh and M. Shub, namely that
stable ergodicity is $C^1$-dense among partially hyperbolic diffeomorphisms with 2-dimensional center bundle \cite{rhrhtu2009}.\par
The aim of this paper is to obtain conservative diffeomorphisms admitting blenders near conservative diffeomorphisms with a pair of hyperbolic periodic points with co-index one. This result is crucial in the proof of the Pugh-Shub Conjecture \cite{rhrhtu2009}. Let us remark that, in fact, what we need for the proof of the Pugh-Shub Conjecture is a very special case of Theorem \ref{main.theorem}. We think that in view of the new importance of blenders for the conservative setting, it is interesting to state the result in its full generality.\par
The main result of this paper is the following.
\begin{theorem}\label{main.theorem} Let $f$ be a $C^r$ diffeomorphism preserving a smooth measure $m$ such that $f$ has two hyperbolic periodic points $p$ of index $(u+1)$ and $q$ of index $u$. Then there are $C^r$ diffeomorphisms arbitrarily $C^1$-close to $f$ which preserve $m$ and admit a $cu$-blender associated to the analytic continuation of $p$.
\end{theorem}
The proof of Theorem \ref{main.theorem} closely follows the scheme in \cite{bonatti-diaz2006}.
%----------------------------------------------------------
\section{Sketch of the proof}\label{section.sketch}
%----------------------------------------------------------
All preliminary concepts are in Section \ref{section.preliminaries}. Let $f$ be a $C^r$ diffeomorphism preserving a smooth measure $m$ such that $f$ has two hyperbolic periodic points $p$ of index $(u+1)$ and $q$ of index $u$. A first step is to prove:
\begin{proposition}\label{proposition.heterodimensional.cycle} Let $f$ be a $C^r$ conservative diffeomorphism such that $f$ has two hyperbolic periodic points $p$ of index $(u+1)$ and $q$ of index $u$, then $C^1$-close to $f$ there is a $C^r$ conservative diffeomorphism such that the analytic continuations of $p$ and $q$ form a co-index one heterodimensional cycle.
\end{proposition}
The proof of this involves a combination of recurrence results by C. Bonatti and S. Crovisier, the Connecting Lemma, and a recent result by A. \'Avila concerning approximation of conservative $C^1$ diffeomorphisms by smooth conservative diffeomorphisms. We prove this proposition in Section \ref{section.real.central.eigenvalues}. \par
The goal of of the rest of the paper is to reduce this heterodimensional cycle to a standard form, in which perturbations are easily made. Note that in a co-index one heterodimensional cycle associated to the points $p$ and $q$ with periods $\pi(p)$ and $\pi(q)$, both $Df^{\pi(p)}(p)$ and $Df^{\pi(q)}(q)$ have $s$ contracting eigenvalues and $u$  expanding eigenvalues. There is a remaining (central) eigenvalue which is expanding for $Df^{\pi(p)}(p)$ and contracting for $Df^{\pi(q)}$. This center eigenvalue, however, could be complex or have multiplicity bigger than one. This could complicate getting a simplified model of the cycle. Next step is to show, like in \cite{bonatti-diaz2006}, the following:
\begin{theorem}\label{teo.real.central.eigenvalues} Let $f$ be a $C^r$ conservative diffeomorphism having a co-index one cycle associated to the periodic points $p$ and $q$. Then $f$ can be $C^1$-approximated by $C^r$ conservative diffeomorphisms having co-index one cycles with real central eigenvalues associated to the periodic points $p'$ and $q'$ which are homoclinically related to the analytic continuation of $p$ and $q$.
\end{theorem}
The fact that $p'$ and $q'$ are homoclinically related to the analytic continuation of $p$ and $q$ is important due to the fact that if ${\rm Bl}^{cu}(p)$ is a $cu$-blender associated to $p$ and $p'$ is homoclinically related to $p$ then ${\rm Bl}^{cu}(p)$ is also a $cu$-blender associated to $p'$. See Remark \ref{remark.blender.homoclinic}.
The proof of Theorem \ref{teo.real.central.eigenvalues} is in Section \ref{section.real.central.eigenvalues}. \newline \par
The proof of Theorem \ref{main.theorem} shall now follow after proving Theorems \ref{teo.strong.homoclinic.intersection} and \ref{teo.blenders} below.
\begin{theorem}\label{teo.strong.homoclinic.intersection} Let $f$ be a $C^r$ conservative diffeomorphism having a co-index one cycle with real central eigenvalues. Then $f$ can be $C^1$-approximated by $C^r$ conservative diffeomorphisms having strong homoclinic intersections associated to to a hyperbolic periodic point with expanding real central eigenvalue. \end{theorem}
Theorem \ref{teo.strong.homoclinic.intersection} is the more delicate part. Next theorem is now standard after C. Bonatti, L. D\'{\i}az and M. Viana's work \cite{bonattidiazviana1995}:
\begin{theorem}\label{teo.blenders} Let $f$ be a $C^r$ diffeomorphism preserving $m$ with a strong homoclinic intersection associated to hyperbolic periodic point $p$ with expanding real center eigenvalue. Then $f$ can be $C^1$-approximated by a $C^r$ diffeomorphism preserving $m$ and having a $cu$-blender associated to $p$.
\end{theorem}
In the creation of conservative blenders, we closely follow the scheme in \cite{bonatti-diaz2006}. It will be
proved that after a $C^1$-perturbation, we obtain a $C^r$ conservative
diffeomorphism such that the co-index one cycle with real central
eigenvalue has local coordinates where the dynamics of the cycle is
affine and partially hyperbolic, with one dimensional central
direction. This is called a {\em simple cycle}, see Definition \ref{definition.simple.cycle}. Let us note that, unlike in \cite{bonatti-diaz2006}, we cannot use Sternberg's Theorem to linearize, due to the obvious resonance. We use instead the Pasting Lemmas by A. Arbieto and C. Matheus \cite{arbieto-matheus2006}. The construction of simple cycles $C^1$-close co-index one cycles with real center eigenvalues is proved in Section \ref{section.simple.cycle}. \par
We shall afterwards produce a continuous family of perturbations $\{f_t\}_{t>0}$ of $f$ shifting the unstable manifold of $q$ in $U_p$ so that it does not intersect $W^s(p)$, thus breaking the cycle. These perturbations preserve the bundles $E^{ss}$, $E^c$ and $E^{uu}$. In this form, they induce maps of the interval on $E^c$. By carefully choosing a small parameter $t>0$, and possibly ``touching" the center expansion/contraction, we may obtain that there is an $E^{ss}\oplus E^{uu}$ plane that is periodic by two different itineraries of $f_t$. Now, the dynamics $f_t$ on this periodic plane is hyperbolic, using a Markovian property, we get two periodic points that are homoclinically related within this plane. Using the $\lambda$-lemma we obtain a periodic point with a strong homoclinic intersection. This is done in Section \ref{section.strong.homoclinic.intersection}.\par
In this way we can apply the well-known techniques of Bonatti, D\'{\i}az and Viana \cite{bonattidiazviana1995} and of Bonatti, D\'{\i}az which give blenders near points with strong homoclinic intersections \cite{bonattidiaz1996}.
Note that these perturbations can be trivially made so that the resulting diffeomorphism be $C^r$ and conservative.
%

%----------------------------------------------------------
\section{Preliminaries}\label{section.preliminaries}
%----------------------------------------------------------
From now on, we shall consider a smooth measure $m$ on a smooth manifold $M$, and a $C^r$ diffeomorphism $f:M\to M$ preserving $m$.
%
%-------------------------------------------------
\subsection{Definitions}
%-------------------------------------------------
We shall say that $f$ is {\de conservative} if $f$ preserves $m$. Given a hyperbolic periodic point $p$ of $f$, the {\de index} of $p$ is number of expanding eigenvalues of $Df^{\pi(p)}(p)$, counted with multiplicity, where $\pi(p)$ is the period of $p$.
\begin{definition}[Heterodimensional cycle]\label{definition.heterodimensional.cycle}
A diffeomorphism $f$ has a heterodimensional cycle associated to
two hyperbolic periodic points $p$ and $q$ of $f$ if their indices are different, and
the stable manifold $W^s(p)$ of $p$ meets the unstable manifold $W^u(q)$ of $q$, and the unstable manifold $W^u(p)$ of $p$ meets the stable manifold $W^s(q)$ of $q$.
\end{definition}
When the indices of $p$ and $q$ differ in one, we say $p$ and $q$ are a {\de co-index one} heterodimensional cycle, or co-index one cycle. \par
We shall say that $f$ is {\de partially hyperbolic} on an $f$-invariant set $\Lambda$, or $\Lambda$ is {\de partially hyperbolic} for $f$ if there is a $Df$-invariant splitting $T_\Lambda=E^s_\Lambda\oplus E^c_\Lambda\oplus E^u_\Lambda$ such that for all $x\in \Lambda$ and all unit vectors $v^\sigma\in E^\sigma$, $\sigma=s,c,u$ we have:
$$\|Df(x)v^s\|<\|Df(x)v^c\|<\|Df(x)v^u\|$$
for some suitable Riemannian metric on $M$. We require that both $E^s$ and $E^u$ be non-trivial.\par
Our goal is to produce a $C^1$-perturbation admitting a blender. Here is the definition of blender we shall be using:
\begin{definition}[$cu$-blender near $p$]\label{definition.blender}
Let $p$ be a partially hyperbolic periodic point for $f$ such that $Df$ is expanding on $E^c$ and $\dim E^c=1$. A small open set ${\rm Bl}^{cu}(p)$, near $p$ but not necessarily containing $p$, is a {\de $cu$-blender associated to $p$} if:
\begin{enumerate}
\item every $(u+1)$-{\de strip well placed} in ${\rm Bl}^{cu}(p)$ transversely intersects $W^s(p)$.
\item This property is $C^1$-robust. Moreover, the open set associated to the periodic point contains a uniformly sized ball.
\end{enumerate}
A $(u+1)$-{\de strip} is any $(u+1)$-disk containing a $u$-disk $D^{uu}$, so that $D^{uu}$ is centered at a point in ${\rm Bl}^{cu}(p)$, the radius of $D^{uu}$ is much bigger than the radius of ${\rm Bl}^{cu}(p)$, and $D^{uu}$ is almost tangent to $E^u$, i.e. the vectors tangent to $D^{uu}$ are $C^1$-close to $E^u$. A $(u+1)$-strip is {\de well placed} in ${\rm Bl}^{cu}(p)$ if it is almost tangent to $E^c\oplus E^u$. See Fgure \ref{figure.blender}.
\end{definition}
Naturally, it makes sense to talk about robustness of these properties and concepts, since there is an analytic continuation of the periodic point $p$ and of the bundles $E^s$, $E^c$ and $E^u$. We can define $cs$-blenders in a similar way. For $cs$-blenders we will consider a partially hyperbolic point such that $E^c$ is one-dimensional and $Df$ is contracting on $E^c$. \par
\begin{figure}[h]
 \begin{minipage}{6.5cm}
 \psfrag{s}{\footnotesize $W^{ss}(p)$}\psfrag{p}{\footnotesize $p$}\psfrag{u}{\footnotesize $W^{uu}(p)$}\psfrag{c}{\footnotesize $W^c(p)$}\psfrag{b}{\footnotesize ${\rm Bl}^{cu}(p)$}
 \includegraphics[width=6cm]{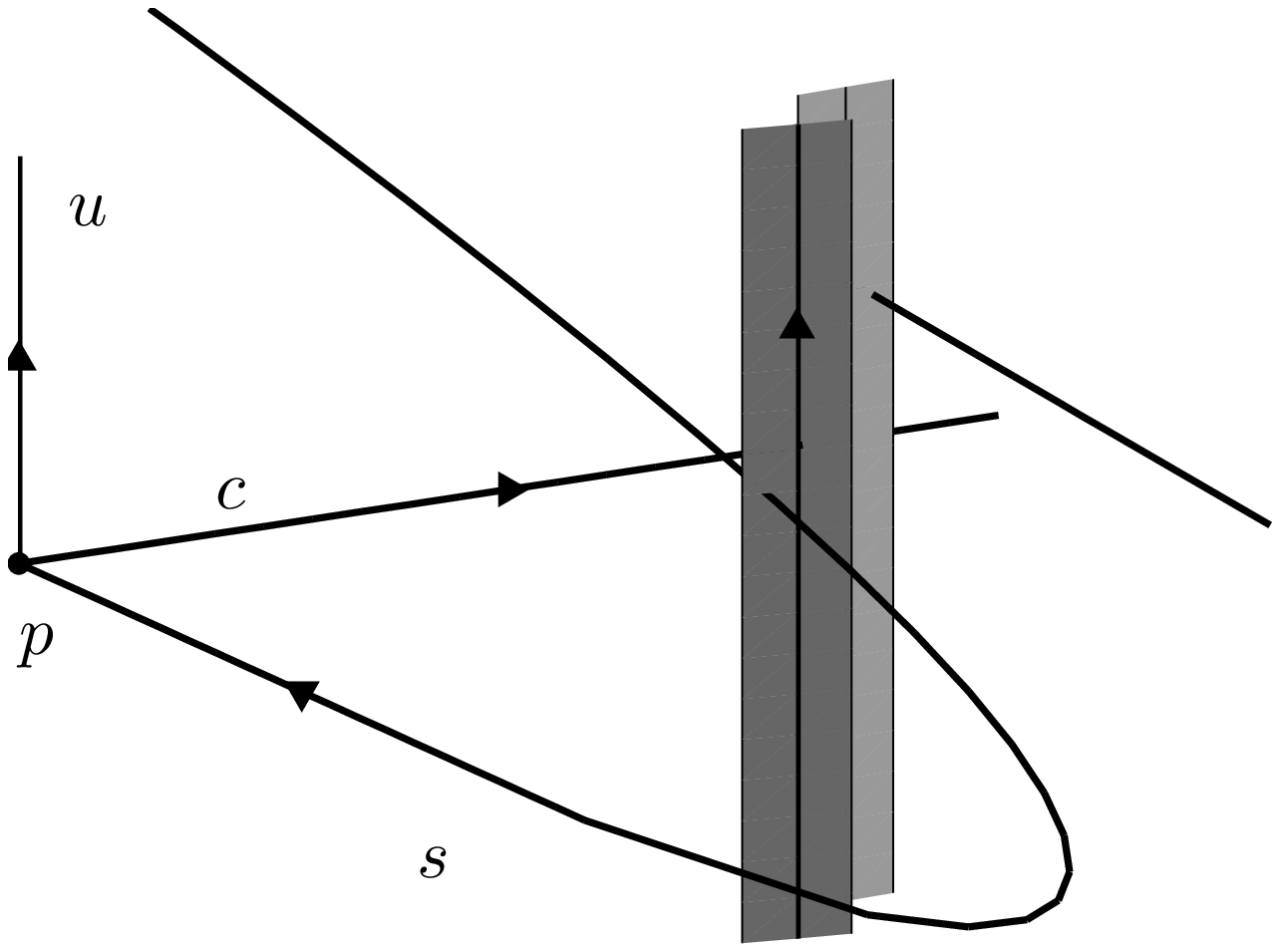}
 \end{minipage}
 \begin{minipage}{6.5cm}
 \psfrag{s}{$W^s(p)$}\psfrag{p}{$p$}\psfrag{u}{$W^u(p)$}\psfrag{c}{$W^c(p)$}\psfrag{b}{${\rm Bl}^{cu}(p)$}
 \includegraphics[width=6cm]{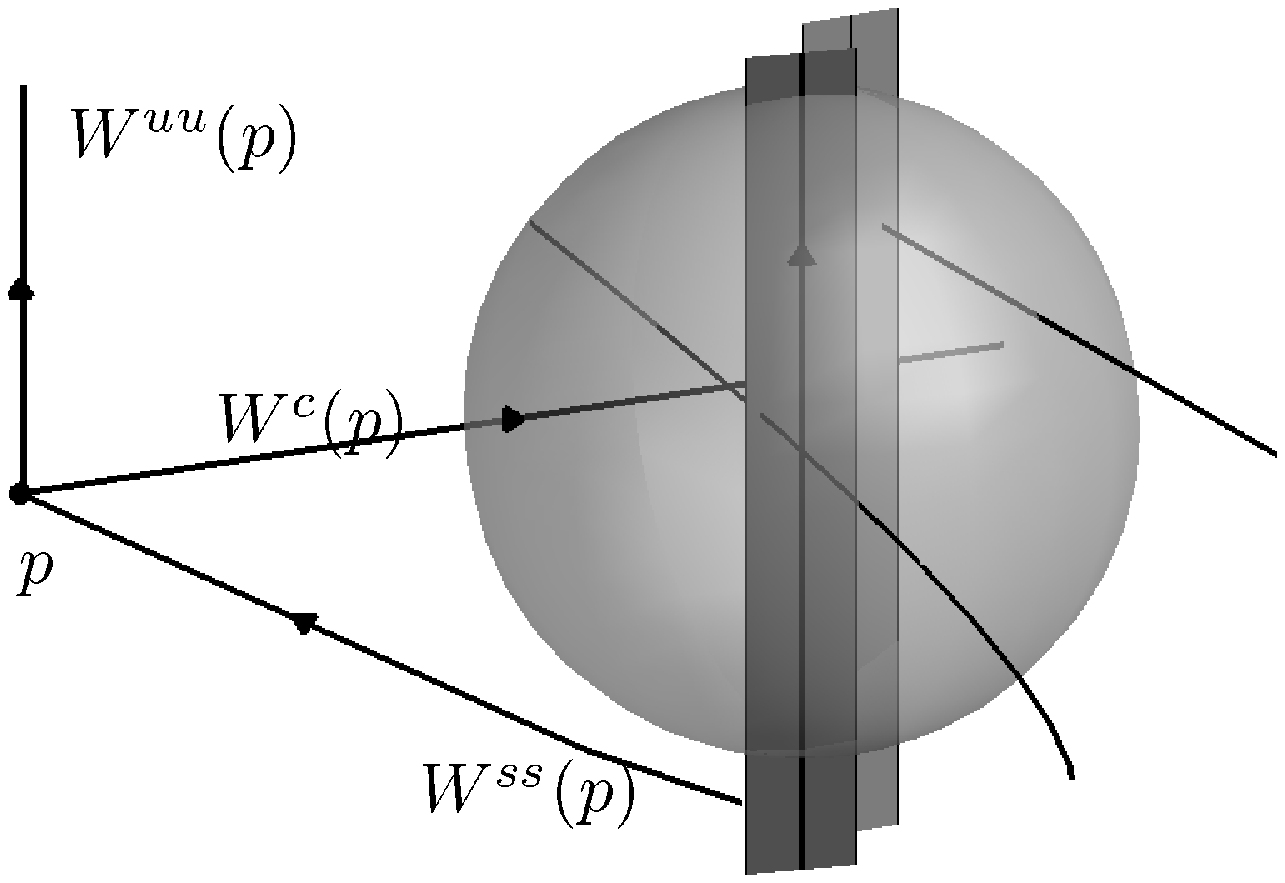}
 \end{minipage}
 \caption{\label{figure.blender} $cu$-blender associated to $p$}
 \end{figure}
This is the definition of blender we shall be using in this work, and, in particular
what we obtain in Theorem \ref{main.theorem}.
We warn the reader that there are other definitions of blenders.
In \cite{bonattidiazvianalibro}, Chapter 6.2. there is a complete presentation
on the different ways of defining these objects. Our definition corresponds to Definition 6.11 of \cite{bonattidiazvianalibro} (the operational viewpoint). In some works of C. Bonatti and L. D\'{\i}az, see for instance \cite{bonattidiaz1996} and \cite{bonatti-diaz2006}, what we call ${\rm Bl}^{cu}(p)$ is known as the {\em characteristic region of the blender}; and what they call {\em blender} is in fact a hyperbolic set which is the maximal invariant set of a small neighborhood of $q$ (Definition 6.9. of \cite{bonattidiazvianalibro}). However, let us remark that, under the hypothesis of Theorem \ref{main.theorem}, we obtain a $C^1$-perturbation preserving $m$ and admitting blenders also in the sense of Definition 6.9 of \cite{bonattidiazvianalibro}, and of \cite{bonattidiazviana1995}. The existence of blenders in the sense of \cite{bonattidiazviana1995} implies the existence of blenders in the sense of Definition \ref{definition.blender}.
\begin{definition}[$cu$-blender associated to $p'$]\label{definition.blender.homoclinic}
Let $p'$ be a partially hyperbolic periodic point for $f$ such that $Df$ is expanding on $E^c$, with $\dim E^c=1$. A small open set $B$ is called {\de $cu$-blender associated to $p'$} if $B={\rm Bl}^{cu}(p)$, where $p$ is a partially hyperbolic periodic point homoclinically related to $p$ and ${\rm Bl}^{cu}(p)$ is a $cu$-blender near $p$. We shall also denote ${\rm Bl}^{cu}(p')$ the $cu$-blender associated to $p'$.
\end{definition}
\begin{remark}\label{remark.blender.homoclinic}
Let us note that if ${\rm Bl}^{cu}(p')$ is a $cu$-blender associated to a hyperbolic periodic point $p'$ and ${\rm Bl}^{cu}(p)$ is a $cu$-blender near $p$, where $p$ is homoclinically related to $p'$, then it follows from the $\lambda$-lemma that $W^s(p')$ transversely intersects every $(u+1)$-strip well placed in ${\rm Bl}^{cu}(p)$, and that this property is robust. The $\lambda$-lemma also implies that a $cu$-blender associated to $p'$ is also a $cu$-blender associated to $p''$ if $p'$ and $p''$ are homoclinically related.  
\end{remark}
\begin{remark}
If $f$ is a partially hyperbolic diffeomorphism in $M$, then we can define $cu$-blender associated to a periodic point $p'$ directly as in Definition \ref{definition.blender}, without requiring that ${\rm Bl}^{cu}(p')$ be near $p'$. Indeed, this requirement is only needed to guarantee the existence of the $(u+1)$-strips well placed in ${\rm Bl}^{cu}(p')$. In the case that $f$ is partially hyperbolic, the $E^u$ and $E^c$ bundles are globally defined, and so Definition \ref{definition.blender} makes sense even without asking that ${\rm Bl}^{cu}(p')$ be near $p'$.
\end{remark}
Let us consider two hyperbolic periodic points $p$ of index $(u+1)$ and $q$ of index $u$, with periods $\pi(p)$ and $\pi(q)$ respectively. Let us denote by $\lambda^\sigma_i(p)$, $\lambda^\sigma_i(q)$ the eigenvalues of $Df^{\pi(p)}(p)$ and $Df^{\pi(q)}(q)$ respectively, with $\sigma =s,c,u$, ordered in such a way that:
\begin{equation}\label{equation.real.central.eigenvalues}
|\lambda^s_1(x)|\leq\dots\leq|\lambda^s_s(x)|\leq|\lambda^c(x)|\leq|\lambda^u_1(x)|\leq\dots\leq|\lambda^u_u(x)|
\end{equation}
where the $\lambda^s_i$'s are contracting and the $\lambda^u_i$'s are expanding. Moreover, $\lambda^c(x)$ is expanding for $p$ and contracting for $q$. Note that $\lambda^c(x)$ could be equal to $\lambda^s_s(x)$, and could be a complex eigenvalue.
\begin{definition}[Cycle with real central eigenvalues] A co-index one heterodimensional cycle associated to two hyperbolic periodic points $p$ and $q$ as described above has {\de real central eigenvalues} if $|\lambda^c(p)|<|\lambda^u_1(p)|$ and $|\lambda^s_s(q)|<|\lambda^c(q)|$.
\end{definition}
Note that when a cycle has real central eigenvalues, then $\lambda^c(p)$ and $\lambda^c(q)$ have multiplicity one.
In the remainder of this subsection, we shall work with heterodimensional cycles with real central eigenvalues. In this case, both orbits of $p$ and $q$ admit a partially hyperbolic splitting $TM=E^s\oplus E^c\oplus E^u$, with $\dim E^c=1$, $\dim E^s=s$ and $\dim E^u=u$. We shall denote $W^{ss}(p)$ the {\de strong stable manifold of $p$}, this means, the invariant manifold tangent to $E^s_p$ that has dimension $s$. Analogously we define $W^{uu}(p)$, the {\de strong unstable manifold of $p$}. Note that the unstable manifold of $p$ could be of dimension $(u+1)$, in which case it would contain $W^{uu}(p)$.
\begin{definition}[Strong homoclinic intersection] A partially hyperbolic periodic point $p$ has a {\de strong homoclinic intersection} if $$\{p\}\subsetneq W^{ss}(p)\cap W^{uu}(p)$$
\end{definition}
The first goal is to obtain a $C^1$-perturbation of $f$ admitting a co-index one cycle with some local coordinates such that the dynamics in a neighborhood of the points are affine:
\begin{definition}[Simple cycle]\label{definition.simple.cycle}
A co-index one cycle associated to the periodic points $p$ of index $(u+1)$ and $q$ of index $u$ is called a {\de simple cycle} if it has real central eigenvalues and:
\begin{enumerate}
\item $p$ and $q$ admit neighborhoods $U_p$ and $U_q$ on which the expressions of $f^{\pi(p)}$ and $f^{\pi(q)}$ are linear and partially hyperbolic, with central dimension 1. We denote the coordinates of a point $(x^s,y^c,z^u)_p$ or $(x^s,y^c,z^u)_q$ according to wether it belongs to $U_p$ or $U_q$.
\item There is a quasi-transverse heteroclinic point $(0,0,z_0)_q
\in W^s (p) \cap W^u(q)$ such that $f^l (0,0,z_0)_q=(x_0,0,0)_p$ for some $l>0$,
and a neighborhood $V\subset U_q$ of $(0,0,z_0)_q$ satisfying $f^l (V) \subset U_p$ for which
$$ \mathcal{T}_1 = f^l : V \rightarrow
f^{l} (V)$$ is an affine map preserving the partially hyperbolic splitting, which is contracting in the $s$-direction, expanding
in the $u$-direction, and an isometry in the central direction.

\item There is a point $(0,y_0^+,0)_p\in W^u (p) \pitchfork  W^s (q)$, with $y_0^+>0$ such that $f^r (0,y_0^+,0)_p=(0,y_0^-,0)_q$ with $y^-_0 < 0$ for some $r>0$, and a neighborhood $W\subset U_p$ of $(0,y_0^+,0)_p$, such that $f^r(W ) \subset U_p$ and
$$\mathcal{T}_2 = f^r : W \rightarrow
f^r (W )$$ is an affine map preserving the partially hyperbolic splitting, and is
contracting in the $s$-direction, expanding
in the $u$-direction, and an isometry in the central direction.
\item There is a segment $I =[y_0^+-\eps,y_0^++\eps]_p^c$ contained in $W^u (p) \pitchfork W^s (q)$, such that the form of $f^r(I)$ in $U_q$ is $J =[y_0^--\eps,y^-_0+\eps]^c_q$.
\end{enumerate}
We call the affine maps $\mathcal{T}_1$ and $\mathcal{T}_2$
the {\de transitions} of the simple heterodimensional cycle.
\end{definition}
%--------------------------------------------------------
\subsection{Preliminary results}
%--------------------------------------------------------
In this subsection we state the preliminary results that shall be used in this work. The following, Theorem 1.3 of \cite{boncrov2004}, allows us to approximate conservative diffeomorphisms by transitive diffeomorphisms preserving $m$:
\begin{theorem}[C. Bonatti, S. Crovisier \cite{boncrov2004}]\label{teo.bonatti.crovisier} There exists a residual set of the set of diffeomorphisms preserving $m$ such that all diffeomorphisms in this set is transitive. Moreover, $M$ is the unique homoclinic class.
\end{theorem}
We shall use this theorem in combination with the Connecting Lemma below:
\begin{theorem}[Connecting Lemma \cite{arnaud2001}, \cite{boncrov2004}]\label{connecting.lemma}
Let $p, q$ be hyperbolic periodic points of a $C^r$
transitive diffeomorphism $f$ preserving a smooth measure $m$. Then, there exists a
$C^1$-perturbation $g\in C^r$
preserving $m$ such that
$W^s(p)\cap W^u(q)\neq \emptyset$.
\end{theorem}
The recent remarkable result by A. \'Avila allows us to approximate $C^1$ conservative diffeomorphisms by $C^\infty$ conservative diffeomorphisms:
\begin{theorem} [A. \'Avila \cite{avila2008}]\label{teo.avila} $C^\infty$ diffeomorphisms are dense in the set of $C^1$ diffeomorphisms preserving $m$.
\end{theorem}
This following conservative version of Franks' Lemma is Proposition 7.4 of \cite{bdp2003}:
\begin{proposition}[Conservative version of Franks' Lemma \cite{bdp2003}]\label{franks.lemma} Let $f$ be a $C^r$ diffeomorphism preserving a smooth measure $m$, $S$ be a
finite set. Assume that $B$ is a conservative
$\eps$-perturbation of $Df$ along $S$. Then for every neighborhood $V$
of $S$ there is a $C^1$-perturbation $h\in C^r$ preserving $m$, coinciding with $f$ on $S$ and out of $V$,
such that $Dh$ is equal to $B$ on $S$.
\end{proposition}
The fundamental tools in order to adapt the construction of C. Bonatti and L. D\'{\i}az \cite{bonatti-diaz2006} to the conservative case are the Pasting Lemmas of A. Arbieto and C. Matheus \cite{arbieto-matheus2006}.
We shall need two such lemmas, one for vector fields and the other for diffeomorphisms. The following is Theorem 3.1. of \cite{arbieto-matheus2006}, and states that we can ``paste" two sufficiently $C^1$-close $C^r$ vector fields, so that one gets the value of the first one on one set and the value of the second one on a disjoint set:

\begin{theorem}[The $C^{1+\alpha}$-Pasting Lemma for Vector Fields]\label{pasting.lemma.para.campos} Given $r>1$ and $ \eps > $ there exists $\delta > 0$ such that if $X,Y \in  \mathcal{X}_m^r (M)$ are two $C^r$ vector fields preserving a smooth measure $m$ that are $\delta$ $C^1$-close on a neighborhood $U$ of a compact set $K$ , then there exist an $m$-preserving vector field $Z \in \cX^r_m (M )$ $\eps$ $C^1$-close to $X$ and two neighborhoods $V$ and $W$ of $K$ such that $K \subset V \subset
U \subset W$ satisfying $Z|_{M\setminus W}=X$ and $X|_V=Y$.
\end{theorem}
We shall also need the Pasting Lemma for diffeomorphisms, which states that we can produce a $C^r$ diffeomorphism by  ``pasting" a conservative diffeomorphism $f$ with its derivative on a neighborhood of a point:
\begin{theorem}[The Pasting Lemma for Diffeomorphisms] \label{pasting.lemma.para.difeos}If $f$ is a $
C^r$ diffeomorphism preserving a smooth measure $m$ and $x$ is a point in $M$, then
for any $ \eps > 0,$ there exists a $C^r$ diffeomorphism $g$ preserving $m$, $\eps$-$C^1$-close to $f$
and two neighborhoods $V$ and $U$ of $x$ such that $x\in V\subset U$ and $ g|_{M\setminus U} = f$ and $g|_{V} =
Df (x)$ (in local charts).
\end{theorem}
%------------------------------------------------------------------
\section{Proof of Theorem \ref{teo.real.central.eigenvalues}}\label{section.real.central.eigenvalues}
%------------------------------------------------------------------
Let $f$ be a $C^r$ diffeomorphism preserving a smooth measure $m$, with two hyperbolic periodic points $p$ of index $(u+1)$ and $q$ of index $u$. Our first step is to show that $f$ can be $C^1$-approximated by a $C^r$ diffeomorphism preserving $m$ such that the analytic continuations of $p$ and $q$ form a co-index one heterodimensional cycle. Namely:
%
%--------------------------------------------------------------
\subsection{Proof of Proposition \ref{proposition.heterodimensional.cycle}}\label{subsection.heterodimensional.cycle}
%--------------------------------------------------------------
Proposition \ref{proposition.heterodimensional.cycle} follows in fact from the more general lemma:
\begin{lemma}\label{lemma.heterodimensional.cycle} Let $f$ be a $C^r$ diffeomorphism preserving a smooth measure $m$ such that $f$ has two hyperbolic periodic points $p$ and $q$, then $f$ is $C^1$-approximated by $C^r$ diffeomorphisms preserving $m$ such that the continuations of $p$ and $q$ either form a heterodimensional cycle in case they have different indices, or else they are homoclinically related.
\end{lemma}
Let us consider the case where $p$ has unstable index $u$ and $q$ has unstable $u'<u$. The case $u=u'$ is simpler, and follows analogously. We shall apply the Connecting Lemma (Theorem \ref{connecting.lemma}) to produce the heterodimensional cycle. However, this Lemma has transitivity as a hypothesis. So, let us consider a conservative diffeomorphism $f_1$ $C^1$-close to $f$ so that $f_1$ is transitive. This $f_1$ exists due to Theorem \ref{teo.bonatti.crovisier} by C. Bonatti and S. Crovisier. Note that $f_1$ is a priori only $C^1$ and has two hyperbolic periodic points $p_1$ and $q_1$ which are the analytic continuations of $p$ and $q$. Now, due to its transitivity, the Connecting Lemma applies, and we can find a $C^1$-close conservative diffeomorphism $f_2$ such that $W^u(p_2)$ intersects $W^s(q_2)$, where $p_2$ and $q_2$ are the analytic continuations of $p$ and $q$. We can even ask that this intersection be transverse (it will be $u-u'$-dimensional). Note that since it is transverse, this intersection persists under $C^1$-perturbations. We will also ask, by using Theorem \ref{teo.bonatti.crovisier} again, that $f_2$ be transitive. Then, we apply the Connecting Lemma again, and obtain a new $C^1$ conservative diffeomorphism $f_3$ $C^1$-close to $f_2$ so that $W^u(p_3)\pitchfork W^s(q_3)\ne \emptyset$ and $y\in W^s(p_3)\cap W^u(q_3)\ne \emptyset$, we can even ask that this last intersection be quasi-transverse, i.e. $T_y W^s(p_3)\cap T_y W^u(q_3)$ does not contain a non-trivial vector. But $f_3$ could be not $C^r$ a priori. Theorem \ref{teo.avila} of A. \'Avila yields a $C^\infty$ conservative diffeomorphism $f_4$ $C^1$-close to $f_3$. Since the stable and unstable manifolds of $p_4$ and $q_4$ vary continuously, we obtain that $W^u(p_4)\pitchfork W^s(q_4)\ne\emptyset$, and $W^s(p_4)$ is close to $W^u(q_4)$ near the point $y$. There is a $C^\infty$ conservative diffeomorphism $f_5$, $C^1$-close to $f_4$, so that $W^s(p_5)$ intersects $W^u(p_5)$. This last perturbation can be made in fact $C^\infty$. This gives the desired heterodimensional cycle. \fp
%------------------------------------------------------------
\subsection{Proof of Theorem \ref{teo.real.central.eigenvalues}}
%------------------------------------------------------------
Theorem \ref{teo.real.central.eigenvalues} follows from Lemma \ref{lemma.real.central.eigenvalues} below and Lemma \ref{lemma.heterodimensional.cycle}. The idea of Lemma \ref{lemma.real.central.eigenvalues} is as in Lemma 4.2. of
\cite{diaz.pujals.ures1999}. See also its generalization to dimension $n$ in Lemmas 1.9 and 4.16 of \cite{bdp2003}. The only difference is that we shall apply the Pasting Lemmas and the Conservative Franks' Lemma instead of the corresponding results.
\begin{lemma}\label{lemma.real.central.eigenvalues} Let $f$ be a conservative $C^r$ diffeomorphism with a hyperbolic periodic point $p$. Then $f$ is $C^1$-approximated by a $C^r$ conservative diffeomorphism with a hyperbolic periodic point $p'$ having the same index as $p$ such that the inequalities in Equation (\ref{equation.real.central.eigenvalues}) are strict.
\end{lemma}
In particular, the eigenvalues of $Df^{\pi(p')}(p')$ are all real and of multiplicity one. \par
Let us assume $p$ is a hyperbolic fixed point for $f$ and let us denote the eigenvalues of $Df(p)$ as in Equation (\ref{equation.real.central.eigenvalues}). We may assume, by using the Conservative Franks' Lemma (Proposition \ref{franks.lemma}) that the multiplicity of all the eigenvalues, complex or real, is one; and that any pair of complex eigenvalues with the same modulus are conjugated, and have rational argument. We shall also assume that $|\lambda^s_s(p)|<|\lambda^c(p)|=|\lambda^u_1(p)|$, and that $\lambda^c(p)$ and $\lambda^u_1(p)$ are complex conjugated expanding eigenvalues. We shall prove that there is $p'$ such that $|\lambda^s_s(p')|<1<|\lambda^c(p')|<|\lambda^u_1(p)|$, whence $p'$ will have index $(u+1)$, just like $p$. In fact, this is all we need, since the same argument applies to show there is $q'$ of index $u$ with real central (contracting) eigenvalue of multiplicity one. Lemma \ref{lemma.real.central.eigenvalues} follows from an inductive argument, since the fact that the eigenvalue is the central one is not used in the argument. \par
Using the Connecting Lemma (Theorem \ref{connecting.lemma}) and genericity arguments we obtain a transverse homoclinic intersection $x\in W^s(p)\pitchfork W^u(p)\setminus\{p\}$. Using the Pasting Lemma for Diffeomorphisms (Theorem \ref{pasting.lemma.para.difeos}) we can linearize $f$ in a small neighborhood $V$ of $p$, so that $f$ remains the same outside a neighborhood $U\supset V$. By considering sufficiently large iterates, we may assume that $x\in W^s_{loc}(p)\cap V$, and $y=f^{-r}(x)\in W^u_{loc}(p)\cap V$ are such that the tangent spaces to $W^s(p)$ at $y$ and to $W^u(p)$ at $x$ are close enough to $E^s_p$ and $E^u_p$. The tangent space to $W^s(p)$ at $x$ and to $W^u(p)$ at $y$ are $E^s_p$ and $E^u_p$ due to the linearization.\par
Birkhoff-Smale provides a periodic point $p'\simeq x$ with period $\pi(p')=n'+r$ where $n'$ is arbitrarily large, such that
\begin{enumerate}
\item $q'=f^{n'}(p')\simeq y$
\item $f^i(p')\in V$ for all $i=0,\dots n'$
\item $p'$ is homoclinically related to $p$
\end{enumerate}
 Let $E^u_{p'}=E^u_p+  p'$ and $E^s_{q'}=E^s_p+q'$. Applying the Conservative Franks Lemma we obtain a perturbation such that
 $$Df^r(q')E^s_{q'}=E^s_{p'}\qquad \mbox{and}\qquad Df^r(q')E^u_{q'}=E^u_{p'}$$
A new perturbation allows us to ``fix" the eigenspace associated to $\lambda^c(p)$ and $\lambda^u_1(p)$,  and obtain normal bases on which the derivative $Df^r(q'):T_{q'}M\to T_{p'}M$ has the form:
$$A=\left(
      \begin{array}{ccc}
        A_s & 0 & 0 \\
        0 & A_c & 0 \\
        0 & 0 & A_u \\
      \end{array}
    \right)
$$
where $A_c$ is a $2\times 2$ matrix acting on a subspace $F^c_{q'}=F^c_p+q'$, where $F^c_p\subset E^u_p$ is the 2-dimensional eigenspace associated to the eigenvalues $\lambda^c(p)$ and $\lambda^u_1(p)$. \par
Applying the Conservative Franks' Lemma again we obtain that

$$Df^{\pi(q')}(q')=\left(
                     \begin{array}{ccc}
                       A^{\pi(q')}_s & 0 & 0 \\
                       0 & A^{\pi(q')}_c & 0 \\
                       0 & 0 & A^{\pi(q')}_u
                     \end{array}
                   \right),\quad A^{\pi(q')}_c=\left(
                                                                \begin{array}{cc}
                                                                  |\lambda^c(p)|^{\pi(q')} & 0 \\
                                                                  0 & |\lambda^c(p)|^{\pi(q')}
                                                                \end{array}
                                                              \right)
$$
It is easy to produce now a perturbation so that the eigenvalues of $F^c_{q'}$ are real and different.\fp
%------------------------------------------------------------------
\section{Creation of Simple Cycles}\label{section.simple.cycle}
%------------------------------------------------------------------
We may assume now that $f$ is a $C^r$ conservative diffeomorphism with a co-index cycle having real central eigenvalues. Our main result will be established if we prove Theorems \ref{teo.strong.homoclinic.intersection} and \ref{teo.blenders}. To do this, we shall perturb that in order to obtain a simplified model of the cycle, that is, a simple cycle. The creation of simple cycles follows the same arguments as in
 \cite[Proposition 3.5]{bonatti-diaz2006} and \cite[Lemma 3.2]{bdpr2003}. The only difference is that we shall use the Pasting Lemmas to linearize in the conservative setting. The goal of this section is to produce a simple cycle.\par
\begin{figure}[h] \label{figure.simple.cycle}
\psfrag{p}{\Small{$p$}}
\psfrag{Esp}{\Small{ $E^c_p$}}\psfrag{Essp}{\Small{ $E^{ss}_p$}}\psfrag{Eup}{ \Small{$E^{uu}_p$}}
\psfrag{Wsp}{\Small{ $W^{s}(p)$}}\psfrag{Wups}{\Small{ $W^{u}(p)$}}
\psfrag{q}{\Small{$q$}}\psfrag{Euq}{\Small{$E^{c}_q$}}\psfrag{Esq}{\Small{$E^{ss}_q$}}\psfrag{Euuq}{\Small{$E^{uu}_q$}}
\psfrag{Wuq}{\Small{$W^{u}(q)$}}\psfrag{Wsq}{\Small{$W^{s}(q)$}}
\includegraphics[width= 7cm]{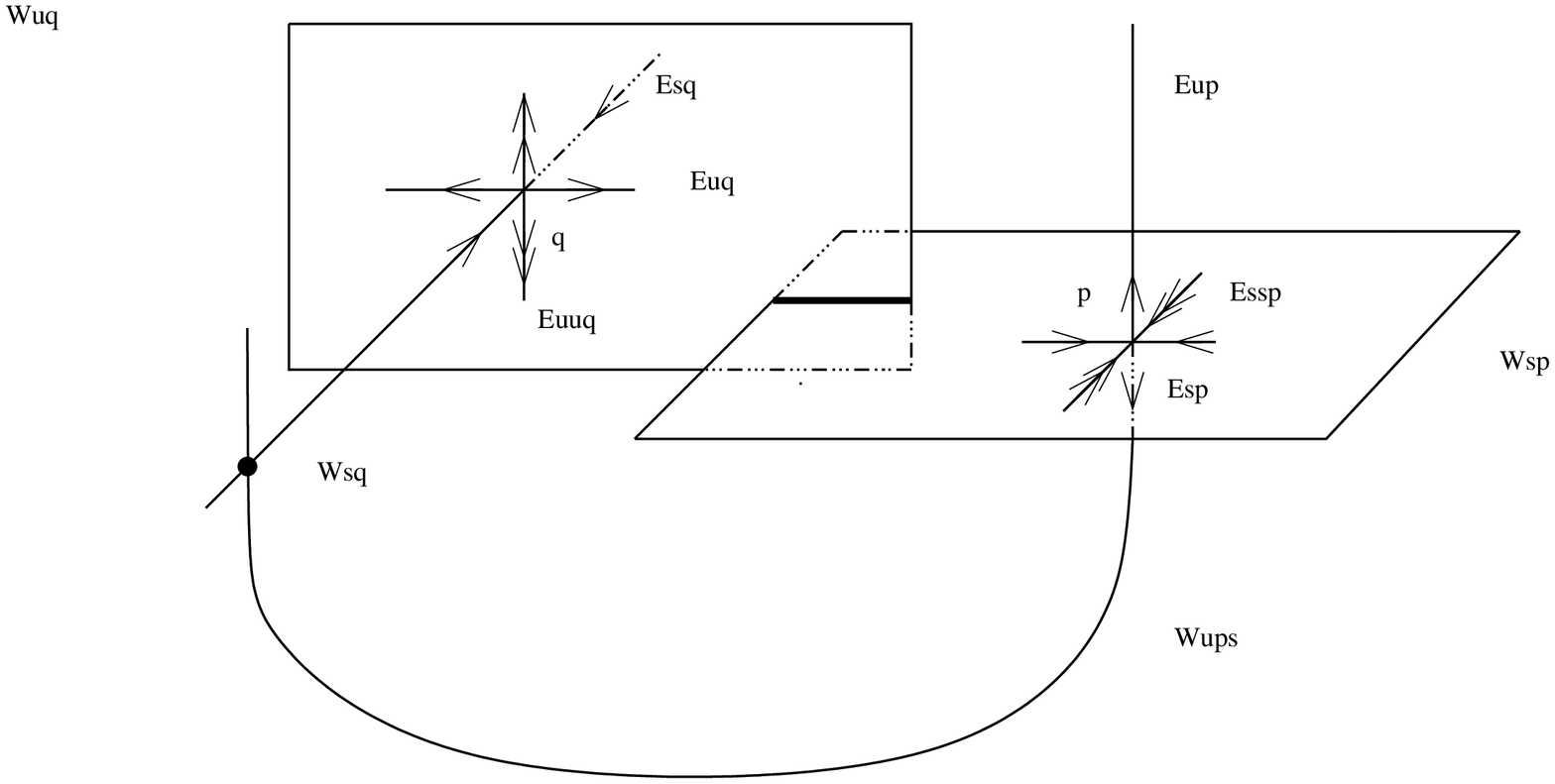}
\caption{}
\end{figure}

Let us suppose that $p$ and $q$ are hyperbolic fixed points of indices $(u+1)$ and $u$ respectively. We may also assume that $W^s(q)$ and $W^u(p)$ have a non-trivial transverse intersection, and $W^s(p)$ and $W^u(q)$ have a point of quasi transverse intersection. On $\{p,q\}$ we have a partially hyperbolic splitting $TM=E^s\oplus E^c\oplus E^u$, where $\dim E^u=u$, $\dim E^c=1$ and $\dim E^s=s$ with $s+u+1=n$. Using the Pasting Lemma for diffeomorphisms, we obtain two neighborhoods $U_p$ and $U_q$ on which we can linearize $f$. We call this new $C^r$ conservative diffeomorphism $g$. $g$ equals $Df(p)$ on $U_p$ and $Df(q)$ on $U_q$, the strong stable and unstable manifolds are, respectively, the $s$- and $u$-planes parallel to $E^s_p$ and $E^u_p$, or $E^s_q$ and $E^u_q$. The center lines parallel to $E^c_p$ and $E^c_q$ are also invariant under $g$. \par
We can choose $g$ so that there is a point of transverse intersection $X$ of $W^u(p)$ and $W^s(q)$, and a point of quasi-transverse intersection $Y$ of $W^s(p)$ and $W^u(q)$. There is a sufficiently large iterate $m>0$ so that $g^{-m}(Y)=(0,0,z_0)_q\in W^u_{loc}(q)$ and $g^m(Y)=(x_0,0,0)_p\in W^s_{loc}(p)$. Take $l=2m$. \par
Now, generically $W^u(p)\pitchfork W^s(q)$ is transverse to the strong unstable in $U_p$ and to the strong stable foliation in $U_q$. So, take a curve $\alpha\subset W^u(p)\pitchfork W^s(q)$ and iterates $m>0$ so large that
$g^{-m}(\alpha)$ is the graphic of a map $\gamma_p:I\to E^{uu}$, and $f^m(\alpha)$ is the graphic of a map $\gamma_q:J\to E^{ss}$. $I$ and $J$ are small segments contained, respectively in $W^c_{loc}(p)$ and $W^c_{loc}(q)$. Note that $g^{-m}(\alpha)$ approaches $W^c_{loc}(p)$ exponentially faster than it approaches $p$; analogously $g^m(\alpha)$ approaches $W^c_{loc}(q)$ exponentially faster than it approaches $q$.
Hence we can choose $m>0$ so large that $\gamma_p$ and $\gamma_q$ and are $C^1$-close to zero.
\par
Let us define $C^r$ vector fields $X_p$ and $X_q$ in suitable neighborhoods of $I\cup \graph(\gamma_p)$ and $J\cup\graph(\gamma_q)$. We define $X_p$ as a vector field that is constant
the hyperplanes parallel to $E^{ss} \oplus  E^{uu}$, such that $X_p(x,y,z)_p=\gamma_p(0,y,0)_p$. That is,
$X_p$ assigns to each point its center coordinate.
Since $X_p$ is constant along the hyperplanes parallel
to $E^{ss} \oplus  E^{uu}$ it is divergence free and it is very
close to the null vector field. Then we can apply the Pasting
Lemma for flows (Theorem \ref{pasting.lemma.para.campos}) and paste $S$ with the null vector field
obtaining a $C^r$-vector field $\bar{X}p$ that is $C^1$-close to the null vector field.
By composing our diffeomorphism with the time-one map of $\bar{X}_p$ we
have a perturbation of $g$ (which we continue to call $g$) such that $g^{-2m}(\graph(\gamma_q))$ is contained in $W^c_{loc}(p)$. Analogously, we obtain a $C^1$-perturbation, so that $g^{2m}(I)\subset W^c_{loc}(q)$.
In this way, we have obtained so far the points $(0,0,z_0)_q$, $(x_0,0,0)_p$, $(0,y_0^+,0)_p$ and $(0,y_0^-,0)_q$.\par
Let us apply the Pasting Lemma for diffeomorphisms again, so that we obtain
a new perturbation for which the transitions $\mathcal{T}_1=g^l|_V$ and $\mathcal{T}_2=g^l|_W$ are affine maps,
where $V$ and $W$ are small neighborhoods of $(0,0,z_0)_q$ and $(0,y_0^+,0)_p$ respectively. We loose no generality in assuming that the images of the hyperplanes $E^{ss}$, $E^{uu}$, $E^{c}$, $E^{ss}\oplus E^c$ and $E^{uu}\oplus E^c$
are in general position. By taking $l>0$ sufficiently large, one obtains that the image of the center-unstable foliation becomes very close to the $E^{uu}\oplus E^c$ in $U_p$. A small perturbation using the Pasting Lemma for vector fields like in the previous paragraph gives us an invariant center-unstable foliation. Indeed, there exists a matrix $A$ with $\det(A)=1$, close to the identity, taking the image of the center-unstable foliation in $W\subset U_p$ into the center-unstable foliation of $g^l(W)\subset U_q$. But now, there exists a vector field $\log A$ such that the time-one map of $\log A$ is $A$. We use the Pasting Lemma for vector fields to paste $\log A$ in a neighborhood of $(0,y_0^-,0)_q$ in $g^l(W)$ with the identity outside of $g^l(W)$. Composing $g$ with the time one map of this vector field, we get a $C^r$ diffeomorphisms $C^1$-close to $g$ such that $g^{l+1}$ leaves the center-unstable foliation invariant. We replace $g^l(W)$ by $g^{l+1}(W)$, and $\cT_2$ by this new affine transition.\par
Let us continue to call $g$ this new $C^r$ conservative diffeomorphism, and $\cT_2$ the new transition. In order to get the invariance of the strong stable foliation, note that by the previous construction, the center unstable foliation is preserved by backward iterations. The backward iterations of the strong stable foliation approach the strong stable foliation in $U_p$. If necessary, we may replace $(0,y_0^+,0)_p$ by a large backward iterate, and $W$ by a corresponding iterate in $U_p$. There is a matrix $B$ with $\det(B)=1$ close to the identity, that preserves the center-unstable foliation and is such that $B\circ g^L|_W$ preserves the strong stable and center-stable foliations. Proceeding as in the previous paragraph, we obtain a $C^1$-close diffeomorphism preserving these foliations. We can repeat now the same argument inside the center-unstable foliation in order to get the invariance of the strong unstable and the center foliations. In this way we obtain an affine partially hyperbolic transition $\cT_2$ preserving $E^{ss}$, $E^{uu}$ and $E^c$. Analogously we obtain $\cT_1$.\par
We only need to show that we can perturb in order to obtain that the transitions $\cT_1$ and $\cT_2$ are isometries on the center foliations. Now we can replace $\cT_1$ by $\cT_1(m_1,m_2)=Dg^{-m_2}(p)\circ \cT_1\circ Dg^{m_1}(q)$ with large $m_1,m_2>0$ on a suitable small neighborhood of $g^{-m_1}(0,0,z_0)_q$. There are infinitely many $m_1,m_2>0$ such that the center eigenvalues of $\cT_1(m_1,m_2)$ are in a bounded away from zero finite interval. Considering $m_1$ and $m_2$ sufficiently large, and changing $(0,0,z_0)_q$ by a point $X$ with coordinates of the same form, we obtain a $C^1$-perturbation in a small neighborhood of the segment of orbit $X$, $g(X),\dots, g^r(X)\in U_q$, where $r=m_1+l+m_2$, such that the action of $\cT_1(m_1,m_2)$ in the central direction is an isometry. The perturbation is
produced using the Pasting Lemma for vector fields as in the previous paragraphs. In analogous way we obtain a transition $\cT_2$ acting as an isometry in the central direction.\par
We have proved the following:
\begin{proposition}
Let $f$ be a $C^r$ conservative diffeomorphism having a co-index one cycle with real central eigenvalues associated to the points $p$ and $q$. Then $f$ can be $C^1$-approximated by $C^r$ conservative diffeomorphisms having simple cycles associated to $p$ and $q$.
\end{proposition}
%---------------------------------------------------------------------------
\section{Proof of Theorem \ref{teo.strong.homoclinic.intersection}}\label{section.strong.homoclinic.intersection}
%---------------------------------------------------------------------------
%
For simplicity we shall assume that $p$ and $q$ are fixed points, and that the
co-index cycle is a simple cycle. We shall call $\lambda_c(p)=\mu$ and $\lambda_c(q)=\lambda$.
We shall assume that $\mu>1$ and $\lambda\in (0,1)$, since no greater complication appears in the cases $\mu<-1$ and $\lambda\in(-1,0)$.\par
Since the cycle is simple, there are coordinates $(x,y,z)_p$ and $(x,y,z)_q$ in suitable neighborhoods of $p=(0,0,0)_p$ and $(0,0,0)_q$ on which the expression of $f$ is:
\begin{equation}\label{equation.linear.expresion.f}A(x,y,z)_p=(A^sx,\mu y, A^uz)_p\qquad\mbox{and}\qquad B(x,y,z)_q=(B^sx,\lambda y, B^uz)_q
\end{equation}
where $A^s$, $B^s$ are contractions, and $A^u$, $B^u$ are expansions. We recall that there are points $(0,0,z_0)_q$ in the quasi-transverse intersection of $W^s(p)\cap W^u(q)$, and $(0,y^+,0)_p$ in the transverse intersection of $W^u(p)\pitchfork W^s(q)$ such that on suitable neighborhoods $V\subset U_q$ and $W\subset U_p$ the transitions $\cT_1=f^l|_V$ and $\cT_2=f^r|_W$ have the form:
\begin{equation}\label{equation.transition1}
\cT_1(x,y,z)_p=(T^s_1x,y+y^--y^+, T^u_1z)_q
\end{equation}
and
\begin{equation}\label{equation.transition2}
\cT_2(x,y,z)_q=(T^s_2x+x_0, y, T^u_2(z-z_0))_p
\end{equation}
where $T^s_1$, $T^s_2$ are contractions and $T^u_1$, $T^u_2$ are expansions.\newline \par
We shall produce a continuous family of perturbations $\{f_t\}_{t>0}$ of $f$ shifting the unstable manifold of $q$ in $U_p$ so that it does not intersect $W^{ss}(p)$, see Figure \ref{figure.perturbation}. These perturbations preserve the bundles $E^{ss}$, $E^c$ and $E^{uu}$.
In this form, they induce maps of the interval on $E^c$. By eventually changing the original $\lambda\in(0,1)$ and $\mu>1$ (so that $f_t$ continues to be conservative), and carefully choosing a small parameter $t>0$, we may obtain that the $E^{ss}\oplus E^{uu}$ plane containing the point $(0,y^+,0)_p$ is periodic by two different (large) itineraries. Now, the dynamics $f_t$ on this periodic plane is hyperbolic, then using a Markovian property, we get two periodic points that are homoclinically related within this plane. The $\lambda$-lemma gives now a periodic point with a strong homoclinic intersection.
\newline\par
Proceeding as in Section \ref{section.simple.cycle} we take a divergence-free vector field $X$ supported in a small neighborhood of $f^{r-1}(V)$, so that the composition $f_t$ of $f$ with the time-$t$ map of $X$ form a $C^1$-family of $C^r$ conservative diffeomorphisms admitting transitions $\cT_2$ of the form (\ref{equation.transition2}) and $\cT_{1,t}$ of the form:
\begin{equation}\label{equation.transition.t}
\cT_{1,t}(x,y,z)_p=\cT_1(x,y,z)_p+(0,t,0)_p
\end{equation}
where $\cT_1$ is as in formula (\ref{equation.transition1})
Since $Df_t(p)=Df(p)$ and $Df_t(q)=Df(q)$, the formulas (\ref{equation.linear.expresion.f}) hold for all small $t>0$. \par
\begin{figure}[h] \label{figure.perturbation}
\psfrag{p}{\Small{$p$}}
\psfrag{Esp}{\Small{ $E^c_p$}}\psfrag{Essp}{\Small{ $E^{ss}_p$}}\psfrag{Eup}{ \Small{$E^{uu}_p$}}
\psfrag{Wsp}{\Small{ $W^s(p)$}}\psfrag{Wups}{\Small{ $W^u(p)$}}
\psfrag{q}{\Small{$q$}}
\psfrag{Euq}{\Small{$E^c_q$}}\psfrag{Esq}{\Small{$E^{ss}_q$}}\psfrag{Euuq}{\Small{$E^{uu}_q$}}
\psfrag{Wuq}{\Small{$W^u(q)$}}\psfrag{Wsq}{\Small{$W^s(q)$}}\psfrag{t}{$t$}
\includegraphics[width= 8cm]{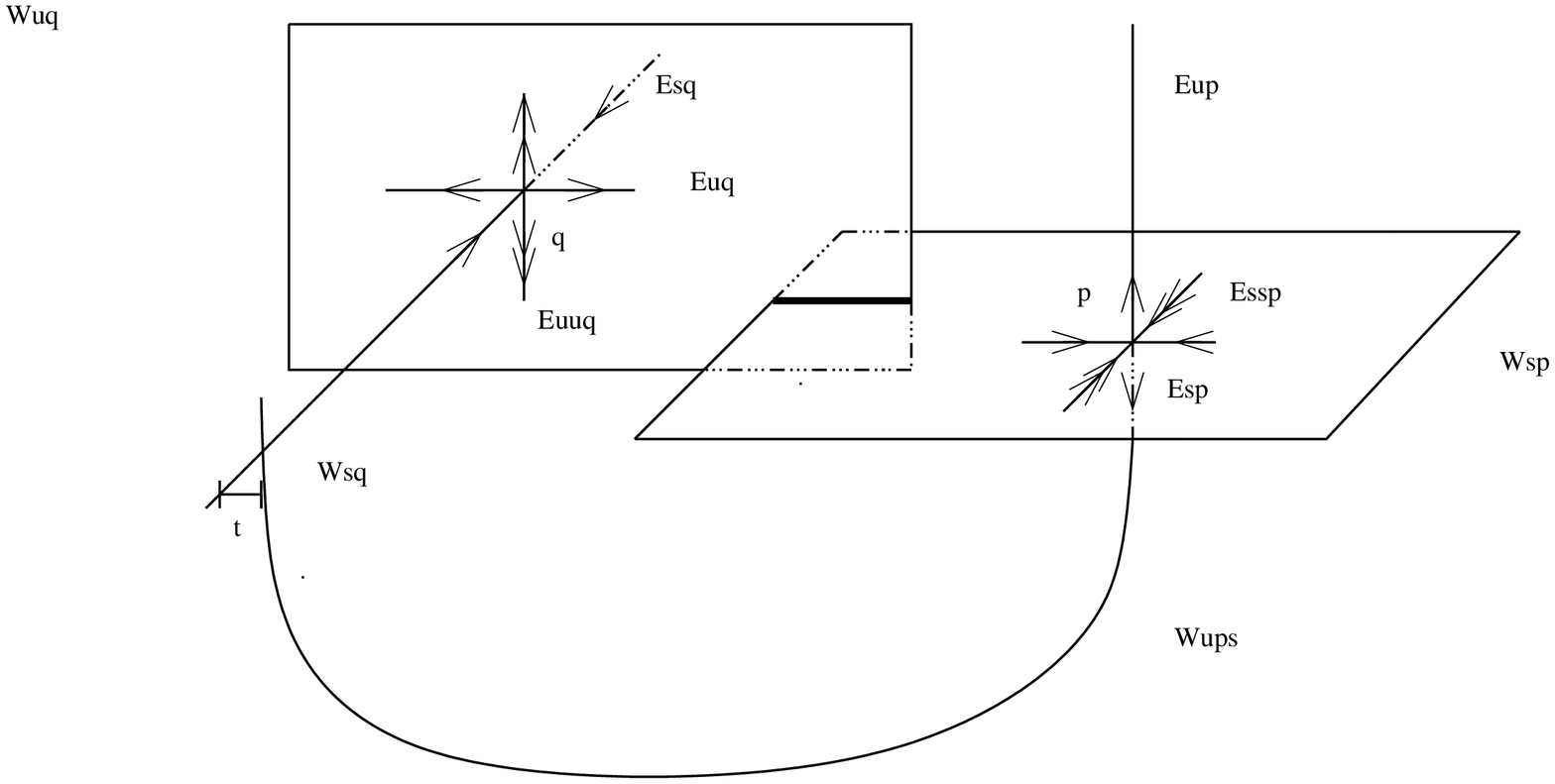}
\caption{}
\end{figure}

Note that if the composition $f_t^n\circ f_t^l\circ f_t^m\circ f_t^r$ makes sense for some point and takes a small neighborhood of $(0,y^+,0)_p$ into $U_p$ then, due to the formulas above, its center coordinate takes the form:
\begin{equation}\label{equation.center.coordinate}
\psi^{m,n}_t(y)=\mu^n\left[\lambda^m(y+\Delta y)+t\right]\qquad\mbox{where}\quad\Delta y=y^--y^+
\end{equation}
Conversely, we have the following:
\begin{lemma}\label{lemma.periodic.points.m.n}
If $m,n,t>0$ are such that $\psi^{m,n}_t(y^+)=y^+$ with sufficiently large $m,n$, then there is a point $$p_{m,n}=(x_{m,n},y^+,z_{m,n})_p$$ in the $E^{ss}\oplus E^{uu}$ plane through $(0,y^+,0)_p$ that is $f_t$-periodic, with period $\pi(p_{m,n})=n+l+m+r$. Its center eigenvalue is $\lambda^c(p_{m,n})=\mu^n\lambda^m$. \par
If there is $(m',n')\ne(m,n)$ for which $\psi^{m',n'}_t(y^+)=y^+$, then there is a point $q_{m,n}$, such that
$$q_{m,n}\subsetneq W^{uu}(p_{m,n})\cap W^{ss}(p_{m,n}).$$
That is, $p_{m,n}$ has a strong homoclinic intersection that is quasi-transverse.
\end{lemma}
\begin{proof}
Let $t>0$ be a small parameter, and let $m,n>0$ be sufficiently large. Suppose that $\psi^{m,n}_t(y^+)=(y^+)$. Then it is easy to see that the $f_t^{n+l+m+r}$ image of the $su$-disc $[-1,1]^s\times\{y^+\}\times[-1,1]^u$ contains a cylinder of the form $$C^u=B^s_\delta((A^s)^nx_0)\times \{y^+\}\times [-1,1]^u,$$ for some sufficiently small $\delta>0$, where $B^s_\delta(x)$ denotes the $s$-disc of radius $\delta$ centered at $x$.\par
Analogously, we obtain that the $f_t^{n+l+m+r}$ pre-image of the $su$-disc $[-1,1]^s\times\{y^+\}\times[-1,1]^u$ contains a cylinder of the form $$C^s=[-1,1]^s\times \{y^+\}\times B^u_\delta(z),$$
for some $z\in (-1,1)^u$ and some suitable small $\delta$, which can be taken equal to the previous one. $B^u_\delta(z)$ denotes the $u$-disc of radius $\delta$ centered at $z$. This implies the existence of a periodic point $p_{m,n}$ of period $\pi(p_{m,n})=n+l+m+r$. The fact that the transitions are isometries on the bentral direction implies that $\lambda^c(p_{m,n})=\lambda^m\mu^n$.\par
If $(m',n')\ne(m,n)$ are such that $\psi^{m',n'}_t(y^+)=y^+$, then the previous argument gives us a periodic point $p_{m',n'}$ which, by construction, is different from $p_{m,n}$. Due to linearity, the unstable manifolds of $p_{m,n}$, $p_{m',n'}$ are, respectively, the $u$-discs $W^{uu}(p_{m,n})=(x_{m,n},y^+)\times [-1,1]^u$ and $W^{uu}(p_{m',n'})=(x_{m',n'},y^+)\times [-1,1]^u$. Also, $W^{ss}(p_{m,n})=[-1,1]^s\times(y^+,z_{m,n})$ and $W^{ss}(p_{m',n'})=[-1,1]^s\times(y^+,z_{m',n'})$. The $u$-discs transversely intersect the $u$-discs in the $su$-plane. Therefore, $p_{m,n}$ and $p_{m',n'}$ are homoclinically related in the $su$-plane, that is
$$W^{uu}(p_{m',n'})\cap W^{ss}(p_{m,n})\ne\emptyset\quad\mbox{and}\quad W^{uu}(p_{m,n})\cap W^{ss}(p_{m',n'})\ne\emptyset$$
The $\lambda$-lemma implies the existence of a point $q\ne p_{m,n}$ in the intersection of the strong stable and strong unstable manifolds of $p_{m,n}$. The expression of $f_t$ in $U_p$ implies that this intersection is quasi-transverse.
\end{proof}
The proof of Theorem \ref{teo.strong.homoclinic.intersection} will be finished after the following:
\begin{lemma}[\cite{bonattidiazviana1995}]
For any $\eps>0$ there are $\mu_0,\lambda_0,t,m,n,n'$ such that $|\lambda_0-\lambda|<\eps$, $|\mu_0-\mu|<\eps$, $|t|<\eps$ and $m,n$ arbitrarily large with $n<n'$ and
\begin{enumerate}
\item $\psi^{m+1,n}_t(y^+)=y^+$
\item $\psi^{m,n'}_t(y^+)=y^+$
\end{enumerate}
\end{lemma}
This is proved in Lemma 3.11 of \cite{bonatti-diaz2006}. The proof of Theorem \ref{teo.strong.homoclinic.intersection} now ends by taking $n,m$ so that $\mu_0^n\lambda_0^{m}>1$. With the techniques used in Section \ref{section.simple.cycle} we can produce a $C^1$ perturbation so that to obtain a $C^r$ conservative diffeomorphism where the linear expressions (\ref{equation.linear.expresion.f}) are such that the center coordinate expansions are, respectively $\lambda_0$ and $\mu_0$. \newline\par
Theorem \ref{teo.blenders} has been proved in \cite{bonattidiazviana1995} and in Theorem 2.1 of \cite{bonattidiaz1996}, see also Section 4.1 of \cite{bonatti-diaz2006}. Note that the perturbations can be trivially made so that the resulting diffeomorphism be $C^r$ and conservative.
%--------------------------------------------------------------------------------------


\begin{thebibliography}{RRRRRR}
%
\bibitem{arbieto-matheus2006} A. Arbieto, C. Matheus, A pasting lemma and some applications for conservative
systems, {\it Erg. Th. \& Dyn. Sys.} {\bf 27}, 5 (2007) 1399-1417.
%
\bibitem{arnaud2001} M.C. Arnaud,Cr\'eation de connexions en topologie {$C\sp 1$},
    {\it Erg. Th. \& Dyn. Systems}, {\bf 21},   2  (2001),339-381.
%
\bibitem{avila2008} A. \'Avila, On the regularization of conservative maps, {\em preprint} (2008) {\tt arXiv:0810.1533}
%
\bibitem{boncrov2004} C. Bonatti, S. Crovisier, R\'ecurrence et g\'en\'ericit\'e,
{\em Inv. Math.}, {\bf 158}, 1 (2004), 33-104.
%
\bibitem{bonattidiaz1996} C. Bonatti, L. D\'{\i}az, Persistent nonhyperbolic transitive
diffeomorphisms. {\em Ann. Math.} (2) {\bf 143} (1996), no. 2,
357-396.
%
\bibitem{bonatti-diaz2006} C. Bonatti, L. D\'{\i}az, Robust
heterodimensional cycles and $C^1$-generic dynamics, {\em Journal of
the Inst. of Math. Jussieu} {\bf 7} (3) (2008) 469-525.
%
\bibitem{bdp2003} C. Bonatti, L. D\'{\i}az, E. Pujals, A $C^1$-generic dichotomy for diffeomorphisms:
Weak forms of hyperbolicity or infinitely many sinks or sources, {\it Ann. Math.} {\bf 158} (2003) 355-418.
%
\bibitem{bdpr2003} C. Bonatti, L. D\'{\i}az, E. Pujals, J. Rocha, Robustly transitive sets and heterodimensional
cycles, {\it Ast\'erisque} {\bf 286} (2003) 187-222.
%
\bibitem{bonattidiazviana1995}  C. Bonatti,  L. D\'{\i}az,  M. Viana,
Discontinuity of the Hausdorff dimension of hyperbolic sets, {\it C.
R. Acad. Sci. Paris S\'{e}r. I Math.} {\bf 320} (1995), no. 6, 713-718.
%
\bibitem{bonattidiazvianalibro} C. Bonatti, L. D\'{\i}az, M. Viana, {\em Dynamics beyond uniform hyperbolicity. A global geometric and probabilistic perspective.} Encyclopaedia of Mathematical Sciences, 102. Mathematical Physics, III. Springer-Verlag, Berlin, 2005.
%
\bibitem{diaz.pujals.ures1999} L. D\'{\i}az, E. Pujals, R. Ures, Partial hyperbolicity and robust transitivity, {\em Acta Math}, {\bf 183} (1999) 1-43.
%
\bibitem{rhrhtu2007} F. Rodriguez Hertz, M. Rodriguez Hertz, A.
Tahzibi, R. Ures, A criterion for ergodicity of non-uniformly
hyperbolic diffeomorphisms, {\it ERA-MS}  {\bf 14}, (2007) 74-81.
%
\bibitem{rhrhtu2009} F. Rodriguez Hertz, M. Rodriguez Hertz, A. Tahzibi, R. Ures, New criteria for ergodicity and non-uniform hyperbolicity, {\em preprint}
%
\end{thebibliography}
\end{document}